\documentclass[11pt,a4paper,reqno]{amsart}

\usepackage[mathscr]{eucal}
\usepackage{color}
\usepackage{amsmath}
\usepackage{amssymb}
\usepackage{amssymb,bbold}
\usepackage{amsfonts}
\usepackage{amscd}
\usepackage{amsthm}
\usepackage{latexsym}
\usepackage{graphicx}
\usepackage{enumitem}

\newtheorem{thm}{Theorem}
\newtheorem{cor}[thm]{Corollary}
\newtheorem{lem}[thm]{Lemma}

\newtheorem{rem}[thm]{Remark}
\newtheorem{ass}{Assumption}

\def\eps{\varepsilon}

\newcommand{\iO}{\int_{\Omega}}
\newcommand{\iT}{\int_{0}^{T}}
\newcommand{\iOT}{\iT\!\!\!\iO}
\newcommand{\itt}{\int_{\tau}^{t}\!\!\iO}

\newcommand{\toto}{\rightrightarrows}

\begin{document}

\title[Weak differentiability in a parabolic equation with hysteresis]{Weak differentiability
of the control-to-state mapping in a parabolic equation with hysteresis}


\author[M. Brokate, K. Fellner, M. Lang-Batsching]{Martin Brokate, Klemens Fellner,
Matthias Lang-Batsching}

\address{Martin Brokate \hfill\break
Center for Mathematical Sciences, Technical University Munich, Boltzmann\-stra{\ss}e 3,
D-85748 Garching b. Munich, Germany}
\email{brokate@ma.tum.de}

\address{Klemens Fellner \hfill\break
Institute of Mathematics and Scientific Computing, University of Graz, Heinrichstra{\ss}e 36,
8010 Graz, Austria}
\email{klemens.fellner@uni-graz.at}

\address{Matthias Lang-Batsching \hfill\break
Center for Mathematical Sciences, Technical University Munich, Boltzmann\-stra{\ss}e 3,
D-85748 Garching b. Munich, Germany}
\email{mala@ma.tum.de}


\begin{abstract}
We consider the heat equation on a bounded domain subject to an inhomogeneous forcing in terms
of a rate-independent (hysteresis) operator and a control variable.

The aim of the paper is to establish a functional analytical setting which allows to prove
weak differentiability properties of the control-to-state mapping. Using results of \cite{BK} and \cite{B}
on the weak differentiability of scalar rate-independent operators, we prove Bouligand and
Newton differentiability in suitable Bochner spaces of the control-to-state mapping
in a parabolic problem.
\end{abstract}

\maketitle

\section{Introduction and Problem formulation}
The aim of this article is to study weak differentiability properties of a parabolic
control problem with a nonlinear operator on the right-hand side, taken
from a class which includes many rate-independent operators.
More precisely, we consider the following problem.

Let $\Omega\subset\mathbb{R}^n$ be a bounded domain with sufficiently smooth boundary
$\Gamma=\partial\Omega$ and denote
$\Omega_T:=\Omega\times(0,T)$ and $\Gamma_T:=\Gamma\times(0,T)$.
Given a control $u\in L^2(\Omega_T)$, we shall consider the following control problem
for the heat equation coupled to an operator $\mathcal{W}$:
\begin{subequations}\label{Problem}
\begin{alignat}{2}
y_t-\Delta y &= u + {\mathcal{W}[y]},&\qquad& \text{in} \quad \Omega_T,\label{Heat}\\
\mathcal{B}[y]&=0,&\qquad& \text{on} \quad \Gamma_T,\label{Boun}\\
y(\cdot,0)&=y_0,&\qquad& \text{on}\quad \Omega. \label{Init}
\end{alignat}
\end{subequations}
Here, $\mathcal{B}$ specifies a mixed Dirichlet-Neumann boundary operator,
which is detailed in Section \ref{sec:2}.

The operator $\mathcal{W}$  is constructed as a space-dependent version
of a scalar operator $\mathcal{V}$, i.e.
\begin{equation}\label{spaceHyst}
\mathcal{W}[y](x,t) = \mathcal{V}[y(x,\cdot)](t) ,\qquad
(x,t)\in\Omega\times [0,T] .
\end{equation}
Thus, $\mathcal{W}$ represents a family of operators acting
on $y(x,\cdot)$, viewed as a function of time, at every $x\in\Omega$.

Concerning the operator $\mathcal{V}$, we assume that
\begin{equation}\label{asshyst.1}
\mathcal{V}: C[0,T] \to C[0,T]
\end{equation}
is a Lipschitz continuous Volterra operator;
more precisely, we require that there exists an $L > 0$ such that
\begin{equation}\label{asshyst.2}
|\mathcal{V}[v](t) - \mathcal{V}[\tilde{v}](t)| \le
L \sup_{0\le s\le t} |v(s) - \tilde{v}(s)|
\end{equation}
holds for every $v,\tilde{v}\in C[0,T]$ and every $t\in [0,T]$.
Condition (\ref{asshyst.2})  implies causality.

The properties (\ref{asshyst.1}) and (\ref{asshyst.2}) are satisfied by many
hysteresis (that is rate-independent Volterra) operators, see
e.g. \cite{BS, Vis, MR}.

It is well known, see \cite{Vis} and Theorem \ref{ibvp.well} in Section \ref{sec:2} below,
that the problem (\ref{Problem}) has a unique solution
for any given $u\in L^2(\Omega_T)$, and that the control-to-state operator
\[
y = Su \,,\qquad
S: L^2(\Omega_T) \to H^1(0,T;L^2(\Omega)) \cap L^\infty(0,T;V) \,,
\]
is well-defined. Here $V$ is some variant of $H^1$ according to the boundary conditions,
see Section \ref{sec:2} for the details.

Assume for a moment that $S$ is Fr\'echet differentiable
w.r.t. suitable norms. Then, for an increment $h\in L^2(\Omega_T)$, we would have
\begin{equation*}
S(u+h) = Su + S'(u)h + o(\|h\|) \,,
\end{equation*}
where the first order approximation $d = S'(u)h$ to the difference $S(u+h)-Su$
depends linearly upon $h$ and is expected to solve a linear problem,
obtained from linearising the original problem.

When $\mathcal{V}$ is a hysteresis operator, $\mathcal{V}$
(and thus $\mathcal{W}$ and $S$)
are not differentiable in the classical sense. Nevertheless, let us
consider the formal linearisation of \eqref{Problem}: 
Given functions $y = Su$ and $h$, we want to determine functions $d$ and $\omega$
as solutions of
\begin{subequations}\label{firstorderProblem}
\begin{alignat}{2}
d_t-\Delta d &= h + \omega,&\qquad& \text{in} \quad \Omega_T,
\\
\omega&=\mathcal{W'}[y;d],&\qquad& \text{in} \quad \Omega_T,
\\
\mathcal{B}[d]&=0,&\qquad& \text{on} \quad \Gamma_T,
\\
d(\cdot,0)&=0,&\qquad& \text{on}\quad \Omega.
\end{alignat}
\end{subequations}
Here, $\omega=\mathcal{W'}[y;d]$ stands for some type of derivative of
$\mathcal{W}$ at $y$ which involves the direction $d$. We do not assume that
the derivative depends linearly on the direction $d$;
indeed, hysteresis operators do not satisfy this property.
Thus, we term the above system 
the \textbf{first order problem};
it is nonlinear whenever the mapping $d\mapsto \omega$ is not linear.
\medskip

Our aim is to derive Bouligand and Newton differentiability of the control-to-state
operator $S$ from the corresponding properties of the operator $\mathcal{V}$
which underlies $\mathcal{W}$.
The notions of Bouligand and Newton differentiability
are closely related, see e.g. \cite{IK} and the definitions at Section 4. Newton differentiability, for instance, is a main prerequisite
in order to guarantee superlinear convergence of the semismooth Newton method
for solving an equation $F = 0$.

In \cite{BK} it was proved that operators $\mathcal{V}$ taken from a certain class
of scalar (that is, the argument of $\mathcal{V}$ is a scalar-valued function)
hysteresis operators is directionally differentiable when considered as
operators from $C[0,T]$ to $L^r(0,T)$ for $1\le r < \infty$.
In \cite{B}, it is shown that $\mathcal{V}$ is Bouligand and Newton differentiable
when considered as an operator from $W^{1,p}(0,T)$ to $L^r(0,T)$ for
$1 < p < \infty$.
%
\medskip

The main result of this paper is the following theorem,
which is detailed with precise assumptions
in Section \ref{sec:bd} (Theorem \ref{sbdiff}).
\medskip

\noindent\textbf{Theorem} (Bouligand and Newton Differentiability).\hfill\\ %
\emph{The control-to-state mapping $u\mapsto y=Su$ is Bouligand resp. Newton
differentiable when considered as an operator}
\begin{equation*}
S: L^{2+\eps}(0,T;L^\infty(\Omega)) \to H^1(0,T; L^2(\Omega)) \cap L^\infty(0,T;V)
\end{equation*}
for sufficiently small $\eps>0$.
Moreover,
the derivative is given by the solution $d$  of the first order problem
\eqref{firstorderProblem}, see also
\eqref{sfirstorderProblem} in Section \ref{sec:firstorder} below.



\begin{rem}
We remark that the results of this paper can be directly generalised to
parabolic problems involving uniform elliptic operators with sufficiently smooth coefficients.
\end{rem}

\medskip
The theorem seems to be of interest for the following reasons.

-- It extends classical sensitivity results (on dependence of a solution of a
differential equation upon parameters) to the case where the right hand side
involves an operator which is not smooth and nonlocal in time.

-- It provides a basis for the use of semismooth Newton methods in such cases.

-- Recently, control problems for partial differential equations with nonsmooth
nonlinearities have received increasing attention. Among others, we want
to point out \cite{CCMW,MS,Mun16,Mun17} and \cite{SWW}. Our result may serve
as an ingredient for obtaining optimality conditions in problems involving
this or a similar state equation.

\medskip
The paper is organized as follows:
In Section \ref{sec:2}, we define precisely the
initial-boundary value problem considered and recall a fundamental existence and
uniqueness result from \cite{Vis}. In Section \ref{sec:3}, we state auxiliary regularity
results for parabolic problems subject to  nonlocal-in-time source terms as appearing in the considered control and first order problems. For the sake of a continuing presentation of the main result we postpone those proofs to Section \ref{sec:Reg}.

In Section \ref{sec:firstorder}, we state the exact differentiability assumptions
for the operator $\mathcal{V}$ and prove existence, uniqueness and regularity for the
first order problem. The proof of the main result, Theorem \ref{sbdiff}, is presented
in Section \ref{sec:bd}.

\section{The control-to-state mapping $S$}\label{sec:2}

In the following, we shall make the statement of Problem (\ref{Problem}) precise.
Let $\Omega\subset\mathbb{R}^n$ be a bounded domain with {sufficiently smooth boundary
$\Gamma=\partial\Omega\in C^{1,1}$} and recall
$\Omega_T:=\Omega\times(0,T)$ and $\Gamma_T:=\Gamma\times(0,T)$.
We consider the problem \eqref{Problem}, i.e.
\begin{subequations}
\begin{alignat*}{2}
y_t-\Delta y &= u + {\mathcal{W}[y]},&\qquad& \text{in} \quad \Omega_T,
\\
\mathcal{B}[y]&=0,&\qquad& \text{on} \quad \Gamma_T,
\\
y(\cdot,0)&=y_0,&\qquad& \text{on}\quad \Omega. 
\end{alignat*}
\end{subequations}
where $u\in L^2(\Omega_T)$ is a given control.
The operator $\mathcal{B}$ specifies a linear boundary condition corresponding to homogeneous
Dirichlet data $\mathcal{B}[y]=y|_{\Gamma_D}=0$ on a subpart of the boundary
$\Gamma_D\subset \Gamma$ with non-zero measure $|\Gamma_D|>0$ and homogeneous
Neumann boundary data
on the remaining part of the boundary $\Gamma_N:=\Gamma\setminus \Gamma_D$, where $|\Gamma_N|=0$ is included.
In the following, we shall use the spaces
$$
V = H^1_{\Gamma_D} = \{ v\in H^1_0 : v|_{\Gamma_D} = 0, \ |\Gamma_D|>0\},
$$
and remark that $V = H^1_0$ in the case $|\Gamma_N|=0$.



The operator $\mathcal{W}$ maps functions on $\Omega_T$ into functions
on $\Omega_T$ according to
\begin{equation*}
\mathcal{W}[y](x,t) = \mathcal{V}[y(x,\cdot)](t) ,\qquad
(x,t)\in\Omega\times [0,T] .
\end{equation*}
As already mentioned in the introduction, the operator $\mathcal{V}$ maps $C[0,T]$ to $C[0,T]$ and
 we assume $\mathcal{V}$ to satisfy the Lipschitz continuity \eqref{asshyst.2}, i.e. that there exists an $L > 0$ such that
\begin{equation}\label{asshyst.12}
|\mathcal{V}[v](t) - \mathcal{V}[\tilde{v}](t)| \le
L \sup_{0\le s\le t} |v(s) - \tilde{v}(s)|
\end{equation}
holds for every $v,\tilde{v}\in C[0,T]$ and every $t\in [0,T]$.
Moreover, we assume the linear growth
\begin{equation}\label{asshyst.13}
|\mathcal{V}[v](t)| \le L \sup_{0\le s\le t} |v(s)| + c_0
\end{equation}
for the same arguments as above and some $c_0 > 0$.

We remark that if one wants to include a space-dependent initial condition for
the hysteresis operator,
one would write $\mathcal{W}[y](x,t) = \mathcal{V}[y(x,\cdot),x](t)$ instead
of (\ref{spaceHyst}); we will not do that in this paper.

The properties (\ref{asshyst.12}) and (\ref{asshyst.13}) carry over to the
operator $\mathcal{W}$ defined in (\ref{spaceHyst}): By denoting
\begin{equation}\label{asshyst.4}
\|y(x,\cdot)\|_{\infty,t} = \sup_{0\le s \le t} |y(x,s)|,
\end{equation}
we immediately obtain for functions $y,\tilde{y}:\Omega\to C[0,T]$ that
\begin{align}\label{asshyst.w.1}
\|\mathcal{W}[y](x,\cdot) - \mathcal{W}[\tilde{y}](x,\cdot)\|_{\infty,t} &\le
L\, \|y(x,\cdot) - \tilde{y}(x,\cdot)\|_{\infty,t},
\\ \label{asshyst.w.2}
\|\mathcal{W}[y](x,\cdot)\|_{\infty,t} &\le L\, \|y(x,\cdot)\|_{\infty,t}
+ c_0,
\end{align}
holds for all $x\in\Omega$ and every $t\in [0,T]$. Thus,
\begin{equation}\label{asshyst.w.3}
\mathcal{W}: L^p(\Omega;C[0,T]) \to L^p(\Omega;C[0,T])
\end{equation}
is well-defined for $1\le p\le \infty$.
%


Under the assumptions above, the following existence and uniqueness result
is a consequence of Theorems X.1.1 and X.1.2 of \cite{Vis}.

\begin{thm}[See {\cite[pp. 297 -- 300]{Vis}}]\label{ibvp.well}
For every $u\in L^2(\Omega_T)$ and every $y_0\in V$,
the initial-boundary value problem given by \eqref{Problem} 
has a unique solution
$$
y\in H^1(0,T;L^2(\Omega)) \cap L^\infty(0,T;V),\qquad
\mathcal{W}[y]\in L^2(\Omega;C[0,T]).
$$
\end{thm}
\begin{proof}
The existence proof is based on the continuous embeddings
$$
H^1(0,T;L^2(\Omega))\cap L^{\infty}(0,T;V) \subset
H^1(\Omega_T) \subset H^{\theta_0}(\Omega;H^{1-\theta_0}(0,T))
$$
for $\theta_0\in(0,1)$ and on the compactness of the embedding
$$
H^{\theta_0}(\Omega;H^{1-\theta_0}(0,T)) \subset L^2(\Omega;C[0,T])\quad \text{for}\quad \theta_0\in(0,1/2).
$$
\end{proof}
\begin{rem}
Given the regularity of the solution stated in Theorem \ref{ibvp.well}, we have furthermore the compact embeddings (see \cite[page 266]{Vis})
\begin{equation*}
H^{\theta_0}(\Omega;H^{1-\theta_0}(0,T)) \subset L^{q_0}(\Omega;C[0,T]),
\qquad 2<q_0 < \frac{2n}{n-2\theta_0}.
\end{equation*}
for all $\theta_0\in(0,1/2)$. Thus, we have also that
\begin{equation}\label{yreg}
y\in L^{q_0}(\Omega;C[0,T]),
\qquad 2<q_0 < \frac{2n}{n-2\theta_0}, \quad \forall \theta_0\in(0,1/2).
\end{equation}
\end{rem}

Using the regularity \eqref{yreg} and the parabolic regularity Lemma \ref{higher} below,
we obtain the following
\begin{cor}\label{cstwelldef}
Let $|\Gamma_N|=0$ and $2\le q < \infty$
or $|\Gamma_N|>0$ and $2\le q < \frac{2n}{n-1}$, then
 the control-to-state operator
\begin{equation}\label{cstdef}
y = Su \,,\qquad
S: L^q(\Omega_T) \to L^q(\Omega;H^1(0,T)) \cap L^\infty(0,T;V) \,,
\end{equation}
is well-defined.
\hfill$\Box$
\end{cor}


\section{Auxiliary parabolic estimates for nonlocal-in-time sources}\label{sec:3}

The following Lemmata \ref{pararega} and \ref{pararegb}
provide parabolic regularity statements for
the heat equation with a nonlocal-in-time source term $g(z)$ satisfying the Lipschitz continuity
property \eqref{asshyst.w.1} that is, the  estimate
\begin{equation}\label{lipsch}
|g(x,t)| \le L\, \sup_{s\le t} | z(x,s) | + f(x,t)
\end{equation}
for a non-negative function $f\ge0$.
We study the following  inhomogeneous parabolic problem:
\begin{subequations}\label{remainderProblem}
\begin{alignat}{2}
z_t-\Delta z &= g,&\qquad& \text{in} \quad \Omega_T,\label{remainHeat}\\
\mathcal{B}[z]&=0,&\qquad& \text{on} \quad \Gamma_T,\label{remainBoun}\\
z(\cdot,0)&=z_0,&\qquad& \text{on}\quad \Omega, \label{remainInit}
\end{alignat}
\end{subequations}
with $z_0\in L^2(\Omega)$ and $g\in L^2(\Omega_T)$.
\medskip

The first example within this paper for a system of the form \eqref{remainderProblem}
with such a function $g$ is the original control problem
\eqref{Problem}, where $g=u+w\in L^2(\Omega_T)$ provided that 
$u\in L^2(\Omega_T)$ which implies $\mathcal{W}[y]\in L^2(\Omega;C[0,T])$ due to
\eqref{asshyst.w.3} and Theorem \ref{ibvp.well}.

The second example is found in the first order system (recall
\eqref{firstorderProblem} or consider \eqref{sfirstorderProblem} below),
where $g = h + \omega\in L^2(\Omega_T)$ provided that $h\in L^2(\Omega_T)$ and thus
$\omega \in L^2(\Omega;L^\infty(0,T))$, see Theorem \ref{ibvp.linwell} below.

\medskip

The results of this section provide a priori estimates for $z$ in terms of $f$.
For the sake of a coherent presentation of our main results, we postpone the proofs of
the following Lemmata \ref{pararega}, \ref{pararegb} and \ref{higher} to Section \ref{sec:Reg}.
The first Lemma \ref{pararega} refines Visintin's
regularity results in Theorem \ref{ibvp.well} by providing explicit a priori estimates.

\begin{lem}[Parabolic regularity I]\label{pararega}
Let $T>0$. Assume $f\in L^2(\Omega_T)$ in \eqref{lipsch} and that additionally $z_0\in H^1(\Omega)$.
Then, the solution to \eqref{remainderProblem} satisfies
\begin{equation}\label{L2five}
\begin{split}
&\iO \sup_{\sigma\le T} |z (x,\sigma)|^2 dx +
\sup_{t\in[0,T]}\iO |\nabla z(t)|^2 \,dx+ \int_{0}^{T}\!\!\iO |z_t|^2 \,dx\,dt
\\ &\qquad \qquad
\le C_1(T)
\left(\int_{0}^{T}\!\!\iO f^2 \,dx\,dt +  \iO  |z_0 (x)|^2 dx
+ \iO \frac{|\nabla z_0|^2}{2} \,dx\right).
\end{split}
\end{equation}
The constant $C_1(T)$ grows at most exponentially in $T$.
\end{lem}

\begin{lem}[Parabolic regularity II]\label{pararegb}
Let $T>0$. Assume $f\in L^1(0,T;L^{\infty}(\Omega))$  in \eqref{lipsch}
and that $z_0\in L^{\infty}(\Omega)$.
Then, the solution to \eqref{remainderProblem} satisfies
\begin{equation}\label{Linfty}
\|z\|_{L^\infty(\Omega_T)}
\le C_2(T) \left(\int_0^T \|f\|_{L^{\infty}_x}(s)\,ds+ \|z_0\|_{L^{\infty}(\Omega)}\right).
\end{equation}
The constant $C_2(T)$ grows at most exponentially in $T$.
\end{lem}

\begin{rem} 
The estimates \eqref{L2five} and \eqref{Linfty}, respectively, imply
the continuity at zero of the mappings
\begin{align*}
&(f,z_0)\in L^2(\Omega_T)\times H^1(\Omega) \mapsto
z \in L^2(\Omega;H^1(0,T))\cap L^{\infty}(0,T;H^1(\Omega))
\intertext{and}
&(f,z_0)\in L^{1}(0,T;L^{\infty}(\Omega))\times
(H^1(\Omega)\cap L^{\infty}(\Omega)) \mapsto z \in L^{\infty}(\Omega_T)
\end{align*}
with bounds which grow at most exponentially in $T$.

\end{rem}

\begin{lem}[Higher regularity]\label{higher} \hfill\\
Assume a smooth boundary operator $\mathcal{B}$ with coefficients in $C^{1,1}$.
Let $z\in L^2(\Omega;H^1(0,T))\cap L^{\infty}(0,T;V)$ be a solution to
the parabolic system \eqref{remainderProblem} with a right-hand-side operator
$g$, which additionally satisfies
for all $2<q<\infty$ that $z\in L^q(\Omega_T)$
implies $g(z)\in L^q(\Omega_T)$.

Then,
if either $|\Gamma_N|=0$ and $2\le q < \infty$
or $|\Gamma_N|>0$ and $2\le q < \frac{2n}{n-1}$, we have
$$
z\in L^q(\Omega;H^1(0,T))\cap L^{\infty}(0,T;V)
\quad\text{and}\quad
z \in L^{q}(\Omega;C[0,T]).
$$
\end{lem}

\section{The first order problem}\label{sec:firstorder}

\subsection*{Bouligand and Newton differentiability.}\hfill\\
Let $X,Y$ be normed spaces, $O\subset X$ open and $F:O\to Y$. If $F$ possesses
a directional derivative $F^{BD}(u;h)$ for all $u\in O$, $h\in X$ with the
property that
\begin{equation}\label{fo.boudef}
\lim_{h\to 0} \frac{\|F[u+h] - F[u] - F^{BD}[u;h]\|}{\|h\|} = 0 ,
\end{equation}
then $F$ is called \textbf{Bouligand differentiable} on $O$ with the
Bouligand derivative $F^{BD}$.

With $X,Y,O,F$ as above, let $\mathcal{L}(X;Y)$ denote the space of linear
continuous mappings $M:X\to Y$. A set-valued mapping $F^{ND}:O\toto \mathcal{L}(X;Y)$
is called a \textbf{Newton derivative} of $F$ in $O$ if
\begin{equation}\label{fo.newdef}
\lim_{h\to 0} \sup_{M\in F^{ND}(u+h)}
\frac{\|F[u+h] - F[u] - Mh\|}{\|h\|} = 0 .
\end{equation}
\medskip

\begin{ass}[Assumptions on $\mathcal{V}$, Bouligand case]\label{ass:Bou}\hfill\\
Let the assumptions \eqref{asshyst.1},
\eqref{asshyst.12} and \eqref{asshyst.13} hold.
Assume further:
\begin{enumerate}[leftmargin=10mm]
\item[(i)]
For every $v,\eta\in C[0,T]$, the limit
\begin{equation}\label{fo.vb.1}
\mathcal{V}^{BD}[v;\eta](t) =  \lim_{\lambda\downarrow 0}
\frac{\mathcal{V}[v + \lambda \eta](t) - \mathcal{V}[v](t)}{\lambda}
\end{equation}
exists and defines a function $\mathcal{V}^{BD}:[0,T]\to\mathbb{R}$.
(Linearity of the mapping $\eta\to \mathcal{V}^{BD}[v;\eta]$ is not assumed.)
\item[(ii)] For every $p\in (1,\infty)$, $r\in [1,\infty)$ and $v\in C[0,T]$
there exists a non-decreasing function
$\rho_{v,p,r}:\mathbb{R}_+\to \mathbb{R}_+$ with $\rho_{v,p,r}(\delta)\to 0$
as $\delta \to 0$ such that
for all $\eta\in W^{1,p}(0,T)$ 
\begin{equation}\label{fo.vb.2}
\begin{split}
&\quad\ \|\mathcal{V}[v+\eta] - \mathcal{V}[v] - \mathcal{V}^{BD}[v;\eta]\|_{L^r(0,T)}
\\ &\qquad\qquad\qquad\qquad\qquad\qquad
\le \rho_{v,p,r}(\|\eta\|_{\infty}) (\|\eta'\|_{L^p(0,T)} + |\eta(0)|) .
\end{split}
\end{equation}
\end{enumerate}
\end{ass}

The play hysteresis operator satisfies Assumption \ref{ass:Bou},
see Theorem 8.2 in \cite{B}.

\begin{ass}[Assumptions on $\mathcal{V}$, Newton case]\label{ass:New}\hfill\\
Assume \eqref{asshyst.1},
\eqref{asshyst.12} and \eqref{asshyst.13}.
Let $\mathcal{V}^{ND}: C[0,T]\toto L(C[0,T];L^\infty(0,T))$ be a set-valued
mapping with the following properties:
\begin{enumerate}[leftmargin=10mm]
\item[(i)]
For every $v,\eta\in C[0,T]$, $M\in \mathcal{V}^{ND}[v]$ and $t\in [0,T]$
we have
\begin{equation}\label{fo.vn.1}
\sup_{s\le t} |(M\eta)(s)| \le L \sup_{s\le t} |\eta(s)| .
\end{equation}
\item[(ii)] For every $p\in (1,\infty)$, $r\in [1,\infty)$ and $v\in C[0,T]$
there exists a non-decreasing function
$\rho_{v,p,r}:\mathbb{R}_+\to \mathbb{R}_+$ with $\rho_{v,p,r}(\delta)\to 0$
as $\delta \to 0$ such that
for all $\eta\in W^{1,p}(0,T)$ and $M\in \mathcal{V}^{ND}[v+\eta]$ 
\begin{equation}\label{fo.vn.2}
\begin{split}
&\quad\ \|\mathcal{V}[v+\eta] - \mathcal{V}[v] - M \eta\|_{L^r(0,T)}
\\ &\qquad\qquad\qquad\qquad\qquad\qquad
\le \rho_{v,p,r}(\|\eta\|_{\infty}) (\|\eta'\|_{L^p(0,T)} + |\eta(0)|) .
\end{split}
\end{equation}
\end{enumerate}
\end{ass}

The play hysteresis operator satisfies Assumption \ref{ass:New},
see Theorem 7.20 in \cite{B}.

\begin{lem}\label{fo.vbd}
Let the Assumptions \ref{ass:Bou} or \ref{ass:New} hold for the Bouligand resp. the
Newton case. Then, $\mathcal{V}^{BD}$ is the Bouligand derivative resp.
$\mathcal{V}^{ND}$ is a Newton derivative for
\begin{equation}\label{fo.vbd.1}
\mathcal{V}: W^{1,p}(0,T) \to L^r(0,T) ,\qquad 1 < p <\infty ,\ 1\le r <\infty.
\end{equation}
In fact, it is possible to choose $\rho_{v,p,r}$ in \eqref{fo.vb.2} resp. \eqref{fo.vn.2} such that $\rho_{v,p,r}\le c_{p,r}$
for some constant $c_{p,r} > 0$
independently of $v$.
Moreover, in the Bouligand case
we have
\begin{equation}\label{fo.vbd.2}
\|\mathcal{V}^{BD}[v;\eta] - \mathcal{V}^{BD}[v;\zeta]\|_{\infty,t}
\le L \|\eta - \zeta\|_{\infty,t} \qquad \text{for all $\eta,\zeta\in C[0,T]$,}
\end{equation}
which yields with
$\mathcal{V}^{BD}[v;0]=0$
\begin{equation}\label{fo.vbd.3}
\|\mathcal{V}^{BD}[v;\eta]\|_{\infty,t} \le
L \|\eta\|_{\infty,t} \,,\qquad \text{for all $\eta\in C[0,T]$.}
\end{equation}
\end{lem}

\begin{proof}
Part (ii) of the Assumptions \ref{ass:Bou} or \ref{ass:New} immediately implies
Bouligand resp. Newton differentiability of $\mathcal{V}$.
The estimate (\ref{fo.vbd.2}) follows from the corresponding estimate for the
difference quotients $(\mathcal{V}[v+\lambda\eta] - \mathcal{V}[v])/\lambda$
due to (\ref{asshyst.12}), passing to the limit $\lambda\to 0$.
Setting either
\begin{gather*}
\xi:=\mathcal{V}[v+\eta] - \mathcal{V}[v] - \mathcal{V}^{BD}[v;\eta]
\\
\intertext{and observing
that $|\xi|\le |\mathcal{V}[v+\eta] - \mathcal{V}[v]| +
|\mathcal{V}^{BD}[v;\eta]|$
or}
\xi:=\mathcal{V}[v+\eta] - \mathcal{V}[v] - M\eta \qquad\text{with}\quad M\in\mathcal{V}^{ND}[v+\eta],
\end{gather*}
respectively, we obtain
from \eqref{asshyst.12} and \eqref{fo.vbd.3} resp. \eqref{fo.vn.1}
the estimate
\begin{align*}
\|\xi\|_{L^r(0,T)}
&\le T^{1/r} \|\xi\|_{\infty,T}\le
2 L T^{1/r} \|\eta\|_{\infty,T}\\
 &\le 2 L T^{1/r} \left(T^{1/p'}\|\eta'\|_{L^p(0,T)}+|\eta(0)|\right),
\end{align*}
which implies the existence of a bound $c_{p,r}$ as claimed.
\end{proof}
\medskip

For the control-to-state mapping $S$, we shall construct the Bouligand derivative $S^{BD}$ resp. a Newton derivative
$S^{ND}$ from the corresponding derivative
of the operator $\mathcal{V}$ appearing in the state system
\[
y_t - \Delta y = u + \mathcal{W}[y] , \qquad
\mathcal{W}[y](x,t) = \mathcal{V}[y(x,\cdot)](t) .
\]
We consider $S: X_S \to Y_S$ with the spaces
\[
X_S = L^{2+\eps}(0,T;L^\infty(\Omega)) , \qquad
Y_S = H^1(0,T;L^2(\Omega)) \cap L^\infty(0,T;V) .
\]
Let $u\in X_S$, $y = S[u]\in Y_S$. Given a variation $h\in X_S$ of the control $u$,
we want to obtain $d\in Y_S$ such that
\begin{equation}\label{fo.fod}
S[u+h] = S[u] + d + o(\|h\|_{X_S})
\end{equation}
as the solution of the first order problem
\begin{subequations}\label{sfirstorderProblem}
\begin{alignat}{2}
d_t-\Delta d &= h + \omega,&\qquad& \text{in} \quad \Omega_T,\label{rlinHeat}\\
\mathcal{B}[d]&=0,&\qquad& \text{on} \quad \Gamma_T,\label{rlinBoun}\\
d(\cdot,0)&=0,&\qquad& \text{on}\quad \Omega. \label{rlinInit}
\end{alignat}
where either
\begin{equation}\label{rlinHystB}
\omega = \mathcal{W}^{BD}[y;d] \qquad \quad \text{in} \quad \Omega_T
\end{equation}
or
\begin{equation}\label{rlinHystN}
\omega = M^W d , \quad M^W\in\mathcal{W}^{ND}[y_h] , \quad y_h = S[u+h]
\qquad \text{in} \quad \Omega_T.
\end{equation}
\end{subequations}
The mappings $\mathcal{W}^{BD}$ and $\mathcal{W}^{ND}$ are specified
in the following; it will turn out that $d$ is the Bouligand derivatives $d = S^{BD}[u;h]$ resp. that the
mappings $M^S$ defined by $d = M^S h$ form a Newton derivative of $S$.

\subsection*{Construction of $\mathcal{W}^{BD}$ and $\mathcal{W}^{ND}$.}\hfill\\
Let $y:\Omega\to C[0,T]$ be measurable.
For the Bouligand case, we define $\mathcal{W}^{BD}[y;d]:\Omega_T\to\mathbb{R}$
for $d:\Omega\to C[0,T]$ by
\begin{equation}\label{fo.wbd.1}
\mathcal{W}^{BD}[y;d](x,t) = \mathcal{V}^{BD}[y(x),d(x)](t).
\end{equation}
For the Newton case, we define
\begin{equation}\label{fo.wnd.1}
\begin{split}
\mathcal{W}^{ND}[y] = &\{M^W|\,M^W:\Omega\to L(C[0,T];L^\infty(0,T)),\,
\\ & \qquad \quad
\text{$M^W(x)\in \mathcal{V}^{ND}[y(x)]$ and (\ref{fo.wnd.2}) holds}\}
\end{split}
\end{equation}
where
\begin{equation}\label{fo.wnd.2}
\begin{split}
(x,t) \mapsto (M^W d)(x,t) := [M^W(x)d(x)](t) \qquad  \\
\text{is measurable for all measurable $d:\Omega\to C[0,T]$.}
\end{split}
\end{equation}
In the following we assume that $\mathcal{W}^{ND}[y]$ is not empty.
Indeed, the play hysteresis operator has this property, see Proposition 9.5 in \cite{B}.

The requirement \eqref{fo.wnd.2} ensures that the function $\omega$ on the right side of
(\ref{rlinHeat}) is measurable in the Newton case; for the Bouligand case (\ref{fo.wbd.1})
no additional assumption is needed.

\begin{lem}\label{fo.wdlip}
Let $y,d_1,d_2:\Omega\to C[0,T]$ be given. Then,
\begin{equation}\label{fo.wdlip.1}
\begin{split}
\|\mathcal{W}^{BD}[y;d_1](x,\cdot) - \mathcal{W}^{BD}[y;d_2](x,\cdot)\|_{\infty,t}
\le L\,\|d_1(x,\cdot) - d_2(x,\cdot)\|_{\infty,t}
\\
\|[M^W d_1](x,\cdot) - [M^Wd_2](x,\cdot)]\|_{\infty,t}
\le L\,\|d_1(x,\cdot) - d_2(x,\cdot)\|_{\infty,t}
\end{split}
\end{equation}
respectively,
holds for all $x\in\Omega$, $t\in [0,T]$ and $M^W\in \mathcal{W}^{ND}[y]$.
As a consequence, for either $\omega =  \mathcal{W}^{BD}[y;d]$ or $\omega = M^W d$, we have (analog to \eqref{fo.vbd.3}) for all $x\in\Omega$
\begin{equation}\label{omegad}
\|\omega(x,\cdot)\|_{\infty,t}
\le L \|d(x,\cdot)\|_{\infty,t}
\end{equation}
and the well-posedness of the mapping
\begin{equation}\label{omegaLp}
d \in L^p(\Omega;C[0,T]) \mapsto \omega\in L^p(\Omega;C[0,T])
\end{equation}
for all $1\le p\le \infty$.
\end{lem}

\begin{proof}
This is an immediate consequence of (\ref{fo.vbd.2}) resp. (\ref{fo.vn.1}).
\end{proof}

We remark that we do not investigate in which function spaces the
mappings $\mathcal{W}^{BD}$ and $W^{ND}$ are actually Bouligand resp.
Newton derivatives of $\mathcal{W}$.
\medskip

\subsection*{Wellposedness of the first order problem.}\hfill\\
The following theorems show that the first order problem is well-posed.
In all of them, we assume that $\mathcal{V}$ satisfies the requirements
specified above in (i) and (ii) for the Bouligand resp. the Newton case.

\begin{thm}\label{ibvp.linwell}
Let the Assumptions \ref{ass:Bou} or \ref{ass:New} hold for the Bouligand resp. the
Newton case.
Let $u,h\in L^2(\Omega_T)$ be given, let $y = Su$, $y_h = S[u+h]$.
Then, the first order problem given by
(\ref{rlinHeat})--(\ref{rlinInit}) and (\ref{rlinHystB}) resp.
(\ref{rlinHystN}) has a unique solution
\[
d\in H^1(0,T;L^2(\Omega)) \cap L^\infty(0,T;V),\qquad
\omega \in L^2(\Omega;L^\infty(0,T)).
\]
\end{thm}
We remark that the function $\omega$ has less regularity than the corresponding
function $\mathcal{W}[y]$ in the original problem
\eqref{Problem}.

\begin{proof}
Due to Lemma \ref{fo.wdlip},
the operators $d \mapsto \mathcal{W}^{BD}[y;d]$ resp. $d \mapsto M^W d$
with $M^W\in \mathcal{W}^{ND}[y_h]$
satisfy the assumptions of Theorems X.1.1 and X.1.2 in \cite{Vis},
which can be extended to cover the range space $L^\infty(0,T)$ instead
of $C[0,T]$ for the operator $\mathcal{W}$.
\end{proof}
\medskip

For the proof of our main result, we shall need
explicit estimates of the regularity stated in the
existence Theorem \ref{ibvp.linwell}. The following
Theorem \ref{ibvp.linest} proves for $h,\omega\in L^2(\Omega_T)$ that parabolic regularity yields
$d\in L^2(\Omega;H^1(0,T))\cap L^{\infty}(0,T;V)$ where we recall that
$L^2(\Omega;H^1(0,T))=H^1(0,T;L^2(\Omega))$.

\begin{thm}\label{ibvp.linest}
Under the assumptions of Theorem \ref{ibvp.linwell}, the solution $d$ of the first order problem (\ref{rlinHeat})--(\ref{rlinInit})
and (\ref{rlinHystB}) resp. (\ref{rlinHystN}) satisfies
\begin{equation}\label{ibvp.linest.1}
\iOT d_t^2 \,dx \,dt + \sup_{t\in [0,T]} \iO |\nabla d|^2 \,dx \le
C_1(T) \iOT h^2 \,dx \,dt.
\end{equation}
{Moreover, if additionally $h\in L^1(0,T;L^{\infty}(\Omega))$, then}
\begin{equation}\label{ibvp.linest.2}
\|d\|_{L^\infty(\Omega_T)} \le C_2(T) \int_0^T \|h(\cdot,t)\|_\infty \,dt \,.
\end{equation}
The constants $C_1(T)$ and $C_2(T)$ do not depend on $h$.

Finally, we have for all $\theta_0\in(0,1/2)$ (and with compact embedding) that
\begin{equation}\label{ibvp.linest.3}
d\in L^{q_0}(\Omega;C[0,T]), \qquad 2<q_0 < \frac{2n}{n-2\theta_0}.
\end{equation}
\end{thm}
\begin{proof}
The proof of \eqref{ibvp.linest.1} follows from estimate \eqref{L2five} in Lemma \ref{pararega}
by setting $z := d$, $g:=h + \omega$  and $f: = |h|$ as well as by noting that \eqref{omegad} implies
\[
|g|(x,t) = |\omega + h|(x,t) \le L \sup_{s\le t} |d(x,s)| + |h(x,t)| \,.
\]
Analogous,  \eqref{ibvp.linest.2} follows
from estimate \eqref{Linfty} in Lemma \ref{pararegb}.
Finally, from the regularity stated in Theorem \ref{ibvp.linwell} (or equally in
\eqref{ibvp.linest.1}), follows
the improved regularity \eqref{ibvp.linest.3} in the same way as \eqref{yreg} from
\cite[page 265-266]{Vis}.
\end{proof}

\begin{cor}\label{ibvp.newton}
Let $u\in L^2(\Omega_T)$ be given, let $y_h = S[u+h]$, $M^W\in\mathcal{W}^{ND}[y_h]$.
Then, the solution mapping $h\mapsto d$ of the first order problem
(\ref{rlinHeat})--(\ref{rlinInit}) and (\ref{rlinHystN}) defines an
element $M^S\in L(X_S;Y_S)$.
\end{cor}

\begin{thm}
Let the Assumptions \ref{ass:Bou} or \ref{ass:New} hold for the Bouligand resp. the
Newton case.
For $2<q<\infty$, consider $u,h\in L^q(\Omega_T)$ and $y = Su$, $y_h = S[u+h]$.
Then, the solution $d$ of the first order problem (\ref{rlinHeat})--(\ref{rlinInit})
and (\ref{rlinHystB}) resp. (\ref{rlinHystN}) satisfies
$$
d\in L^q(\Omega;H^1(0,T)) \cap L^\infty(0,T;V),\qquad
\omega\in L^q(\Omega;L^\infty(0,T)).
$$
\end{thm}
\begin{proof}
The statement follows directly from Lemma \ref{higher} with $g=h+\omega$ and \eqref{omegad} resp. \eqref{omegaLp}.
\end{proof}

\section{Bouligand and Newton differentiability of $S$}\label{sec:bd}

Here we state and prove the main theorem of this paper.

\begin{thm}\label{sbdiff}
Assume that the operator $\mathcal{V}$, which underlies the operator $\mathcal{W}$,
satisfies 
the Assumptions \ref{ass:Bou} or \ref{ass:New}
for the Bouligand resp. the Newton case.
Consider the parabolic hysteresis problem \eqref{Problem}-\eqref{spaceHyst}.

Then, for sufficiently small $\eps>0$ the control-to-state mapping $u\mapsto y=Su$ is
Bouligand resp. Newton differentiable when considered as an operator
\begin{equation}\label{sbndiff.1}
S: X_S = L^{2+\eps}(0,T;L^\infty(\Omega) \to
Y_S = H^1(0,T;L^2(\Omega))\cap L^\infty(0,T;V).
\end{equation}
The Bouligand derivative $d = S^{BD}(u;h)$ is given by the solution of the first order problem
(\ref{rlinHeat})--(\ref{rlinInit}) and (\ref{rlinHystB}).
A Newton derivative $S^{ND}: X_S\toto L(X_S;Y_S)$ is given by
\begin{equation}\label{sbndiff.2}
\begin{split}
S^{ND}[u] &= \{M^S: \text{$d = M^S h$ solves
(\ref{rlinHeat})--(\ref{rlinInit})
with $\omega = M^W d$}  
\\ &\qquad \qquad
\text{for some $M^W\in \mathcal{W}^{ND}[{y}]$}\}.
\end{split}
\end{equation}
\end{thm}

The assumption made above that the sets $\mathcal{W}^{ND}[y]$ are nonempty
ensures that $S^{ND}[u]$ is not empty.

\begin{proof}
We first consider an increment $h\in L^2(0,T;L^\infty(\Omega))$ of a given nominal
control $u\in L^2(0,T;L^\infty(\Omega))$. The restriction to $L^{2+\eps}(0,T;L^\infty(\Omega))$ will not be required
until later in the proof.
We denote by
\[
y = S[u] ,\quad y_h = S[u+h]
\]
the corresponding states. Let $d$ be the solution of the first order problem according
to Theorem \ref{ibvp.linwell}. The remainder
\begin{equation}\label{sbndiff.5}
r_h = y_h - y - d
\end{equation}
solves the system
\begin{subequations}\label{remainder}
\begin{alignat}{2}
(r_h)_t-\Delta r_h &= \mathcal{W}[y_h] - \mathcal{W}[y] - \omega,
&\qquad& \text{in} \quad \Omega_T,\label{rlinHeatr}\\
\mathcal{B}[r_h]&=0,&\qquad& \text{on} \quad \Gamma_T,\label{rlinBounr}\\
r_h(\cdot,0)&=0,&\qquad& \text{on}\quad \Omega. \label{rlinInitr}
\end{alignat}
where either
\begin{equation}\label{rlinHystBr}
\omega = \mathcal{W}^{BD}[y;d] \qquad \qquad \text{in} \quad \Omega_T
\end{equation}
or
\begin{equation}\label{rlinHystNr}
\omega = M^W d , \quad M^W\in\mathcal{W}^{ND}[y_h] ,
\qquad \quad \text{in} \quad \Omega_T.
\end{equation}
\end{subequations}
We want to estimate the right side of (\ref{rlinHeatr}).
From (\ref{asshyst.w.1}) we get
\begin{equation}\label{sbndiff.10}
|\mathcal{W}[y_h] - \mathcal{W}[y+d]|(x,t) \le
L\sup_{s\le t} |y_h - y - d|(x,s) = L\sup_{s\le t} |r_h(x,s)| .
\end{equation}

For the remaining part of the right side of (\ref{rlinHeatr}), we set
\begin{equation}\label{zhf}
f(x,t):=|\mathcal{W}[y+d] - \mathcal{W}[y] - \omega|(x,t)
\ge 0.
\end{equation}
Note that (\ref{fo.vb.2}) resp. (\ref{fo.vn.2})
with $r=2$ yields the estimate
\begin{align}
\int_0^T f^2(x,t)\,dt &= \int_0^T |\mathcal{W}[y+d] - \mathcal{W}[y] - \omega|^2(x,t)\,dt \nonumber \\
&\le
\rho^2_{y(x,\cdot)}(\|d(x,\cdot)\|_{\infty,T}) \cdot
\|d_t(x,\cdot)\|^2_{L^p(0,T)},
\label{sbndiff.11}
\end{align}
where we have suppressed the dependence
of $\rho$ on the integration exponents $2$ and $p\in(1,\infty)$.


In the next step, we use that system \eqref{remainder} satisfies the assumptions
of Lemma \ref{pararega} with estimate
\eqref{lipsch}.
Thus, we have 

\begin{equation}
\int_{0}^{T}\!\!\iO (r_h)_t^2 \,dx\,dt+
\sup_{t\in[0,T]}\iO |\nabla r_h|^2 \,dx
\le C_1(T) \int_{0}^{T}\!\!\iO f^2 \,dx\,dt,\label{L2fivebb}
\end{equation}

We now estimate $f$.
Recalling \eqref{sbndiff.11}, we have
\begin{align*}
\iO \int_0^T  f^2 \,dt \,dx &\le \iO
\biggl(\int_0^T |d_t(x,s)|^p \,ds\biggr)^{\frac{2}{p}}
{{\rho_y^2}}(x,\|d(x,\cdot)\|_{\infty,T}) \,dx.
\end{align*}
By using H\"older's inequality in space with exponent $p$, we continue
to estimate
\begin{align*}
\iO \int_0^T f^2 dx dt
&\le
\biggl(\iO  \biggl(\int_0^T
\!\!|d_t(x,s)|^p \,ds\biggr)^{\!\!2}\!dx\biggr)^{\!\frac{1}{p}} \nonumber \\
&\qquad\qquad\times
\biggl(\iO {{\rho_y}}\bigl(x,\|d(x,\cdot)\|_{\infty,T}\bigr)^{2p'}
dx\biggr)^{\!\frac{1}{p'}}
\\
&\le \|d_t\|^2_{L^{2p}(\Omega_T)}
\biggl(\iO {{\rho_y}}\bigl(x, \|h\|_{L^1(0,T;L^\infty(\Omega))}\bigr)^{2p'}
dx\biggr)^{\!\frac{1}{p'}}\nonumber
\end{align*}
where we have used \eqref{ibvp.linest.2} and the fact that ${{\rho_y}}$ is monotone non-decreasing in the second argument. Therefore
\begin{equation}\label{rrrr}
\iOT f^2 dx dt \le
\|d_t\|^2_{L^{2p}(\Omega_T)}
\, \tilde{\rho}_{y}[h],
\end{equation}
where the remainder term
\begin{equation}\label{rzero}
\tilde{\rho}_{y}[h] :=
\bigl\|{\rho_y}\bigl(x, \|h\|_{L^1(0,T;L^\infty(\Omega))}\bigr)\bigr\|^2_{{L^{2p'}}(\Omega)}\xrightarrow{h\to 0} 0
\end{equation}
tends to zero as $h\to 0$ for all choices $p'<\infty$ by the Lebesgue dominated convergence theorem
since $\rho_y:\Omega\times\mathbb{R}_+\to \mathbb{R}_+$ is
a function with $\rho_y(x,\delta)\to 0$ for all $x\in\Omega$
as $\delta \to 0$, which moreover is bounded independently from $y$,
by assumption (ii) on $\mathcal{V}$, see (\ref{fo.vb.2}),
(\ref{fo.vn.2}) and Lemma \ref{fo.vbd}.

As a consequence, by setting $2p=2+\eps$
and $2p' = 2 + \frac{4}{\eps}$, we aim to prove Bouligand resp. Newton differentiability
of the operator
\begin{equation}\label{sbdiff.2}
S: L^{2+\eps}(0,T;L^\infty(\Omega)) \to H^1(0,T; L^{2}(\Omega)) \cap L^\infty(0,T;V) \,,
\end{equation}
where the space on the right hand side of \eqref{sbdiff.2} corresponds to the regularity
of the left hand side of \eqref{L2fivebb}.

By combining \eqref{L2fivebb} with \eqref{rrrr}, \eqref{rzero},
we are left to prove that
\begin{equation}\label{dh}
\|d_t\|^2_{L^{2+\eps}(\Omega_T)} \le \|h\|^2_{L^{2+\eps}(0,T;L^\infty(\Omega))}.
\end{equation}

In order to prove this estimate, we use that
the Hilbert space parabolic regularity estimate \eqref{L2five} to the first order problem \eqref{sfirstorderProblem} with $d(\cdot,0)=0$, that is
$\|d_t\|_{L^{2}(\Omega_T)} \le C(T) \|h+\omega\|_{L^{2}(\Omega_T)}$, extends also to
$L^q$-spaces with $q=2+\eps$ for sufficiently small $\eps>0$ (see \cite{HDJKR}),
i.e. there exists a constants $C$
\begin{align}\label{done}
\|d_t\|_{L^{q}(\Omega_T)} \le C \|h+\omega\|_{L^{q}(\Omega_T)}
\le C\|h\|_{L^q(0,T;L^\infty(\Omega))} + C\|\omega\|_{L^{q}(\Omega_T)}
\end{align}
Next, we observe that estimate \eqref{omegad}  in Lemma \ref{fo.wdlip} implies for all $1\le q\le\infty$
\begin{equation}\label{Wprimereg2}
\sup_{t\le T}\| |\omega(\cdot,t)|\|_{L^q(\Omega)}
\le L \| \sup_{t\le T} | d(\cdot,t)| \|_{L^q(\Omega)}.
\end{equation}
Using \eqref{Wprimereg2}, we estimate
\begin{align}
\|\omega\|^q_{L^{q}(\Omega_T)} &\le
T \sup_{t\le T} \iO |\omega(\cdot,t)|^q\,dx
\le T L^q \iO \sup_{t\le T} |d(\cdot,t)|^q\,dx\nonumber \\
&\le T L^q |\Omega| \|d\|^q_{L^{\infty}(\Omega_T)}
\le C(T) L^q |\Omega| \left(\int_0^T \|h(\cdot,t)\|_\infty \,dt\right)^q
\nonumber\\
&\le C(T,L,\Omega,q) \|h\|^q_{L^q(0,T;L^\infty(\Omega))},
\label{dtwo}
\end{align}
where the second last estimate is due to \eqref{ibvp.linest.2}.
Combining \eqref{done} and \eqref{dtwo} yields
\begin{align}\label{dthree}
\|d_t\|_{L^{q}(\Omega_T)} \le
C(T,L,\Omega,q)\|h\|_{L^q(0,T;L^\infty(\Omega))},
\end{align}
which proves \eqref{dh} and thus ends the proof of Theorem \ref{sbdiff}.
\end{proof}

\section{Proofs ot the regularity estimates}\label{sec:Reg}

\begin{proof}[Proof of Lemma \ref{pararega}]

First, we prove estimate \eqref{L2five}.
To this end, we test \eqref{remainHeat} with $z_t$ and
integrate over $\Omega_T$. Note that
$z_t\in L^2(\Omega_T)$ due to the existence result Theorem \ref{ibvp.well}.
After integration by parts 
and using \eqref{lipsch}, i.e. $|g(x,t)| \le L\, \sup_{s\le t} | z(x,s) | + f(x,t)$ with $f(x,t)\ge0$,
we obtain for all $0\le \tau< t \le T$
\begin{multline}\label{L2one}
\itt |z_t|^2 \,dx\,ds + \itt \partial_t\left(\frac{|\nabla z|^2}{2}\right) \,dx\,ds
 \le  \itt |g(x,s)| |z_t(x,s)|\,dx\,ds
\\
 \le L \itt \sup_{\sigma\le s} |z (x,\sigma)||z_t(x,s)|\,dxds
 + \itt f |z_t|\,dxds,
\end{multline}
where we remark that all boundary terms vanish for the considered
homogeneous boundary operator $\mathcal{B}$ in \eqref{remainBoun}.
Moreover, we may replace the second term in the first line by
$ \frac{1}{2}\int_{\tau}^{t} \frac{d}{dt}\iO \vert \nabla z \vert^2 dx \,ds$.

In order to handle the first term on the right hand side of \eqref{L2one},
we use that for all $x\in\Omega$
\begin{equation}\label{sup}
\sup_{0 \le \sigma \le s} |z(x,\sigma)| \le \sup_{0 \le \sigma \le \tau} |z(x,\sigma)|
+ \int_{\tau}^{s} |z_t(x,\sigma)| \,d\sigma.
\end{equation}
After inserting \eqref{sup} into the first term on the right hand side of \eqref{L2one},
we estimate, by using Young's inequality twice, 
\begin{align*}
\int_{\tau}^{t}\!&\!\iO \sup_{\sigma\le s} |z (x,\sigma)||z_t(x,s)|\,dx\,ds \\
&\le  \iO \sup_{\sigma\le \tau} |z (x,\sigma)|\int_{\tau}^{t} |z_t(x,s)|\,ds\,dx
+\!
\int_{\tau}^{t}\!\!\iO\int_{\tau}^{s} |z_t(x,\sigma)||z_t(x,s)| \,d\sigma \,dx\,ds
\\
&\le  \frac{1}{L} \iO \sup_{\sigma\le \tau} |z (x,\sigma)|^2\, dx
+ \frac{L}{4} \iO \left(\int_{\tau}^{t} |z_t(x,s)|\,ds\right)^2 \,dx
\\
&\quad + \int_{\tau}^{t}\!\!\iO\int_{\tau}^{s} \frac{|z_t(x,\sigma)|^2}{2}+\frac{|z_t(x,s)|^2}{2}\,
d\sigma \,dx\,ds \\
&\le  \frac{1}{L} \iO \sup_{\sigma\le \tau} |z (x,\sigma)|^2 \,dx
+ \frac{L(t-\tau)}{4} \itt |z_t(x,s)|^2\,dx\,ds
\\
&\quad + (t-\tau)\itt |z_t(x,s)|^2\,dx\,ds.
\end{align*}
Coming back to \eqref{L2one}, we obtain by using Young's inequality once more on the
second term of \eqref{L2one}
\begin{multline}\label{L2two}
\itt |z_t|^2 \,dx\,ds + \iO \frac{|\nabla z(t)|^2}{2} \,dx
\le    \iO \frac{|\nabla z(\tau)|^2}{2} \,dx +  \itt f^2 \,dx\,ds \\
+  \iO \sup_{\sigma\le \tau} |z (x,\sigma)|^2 \,dx
+ \left[(t-\tau)\Bigl(\frac{L^2}{4}+L\Bigr)+\frac{1}{4}\right]  \itt |z_t|^2 \,dx\,ds.
\end{multline}
In order to control the first term in the second line of \eqref{L2two}, we observe first that
\begin{multline}\label{subevol}
 \iO \sup_{\sigma\le t} |z (x,\sigma)|^2 \,dx
\le \iO \left(\sup_{\sigma\le \tau} |z (x,\sigma)|
+ \int_{\tau}^{t} |z_t(x,s)|\,ds\right)^2 \,dx\\
\le 2 \iO \sup_{\sigma\le \tau} |z (x,\sigma)|^2 \,dx
+ 2(t-\tau) \itt |z_t(x,s)|^2\,dx\,ds .
\end{multline}
Thus, combining  \eqref{L2two} and \eqref{subevol} yields
\begin{multline}\label{iterestimate}
\iO \sup_{\sigma\le t} |z (x,\sigma)|^2 \,dx + \iO \frac{|\nabla z(t)|^2}{2} \,dx +
\itt |z_t|^2 \,dx\,ds \\
\le
3 \iO \sup_{\sigma\le \tau} |z (x,\sigma)|^2 dx +  \iO \frac{|\nabla z(\tau)|^2}{2} \,dx +
\itt f^2 \,dx\,ds \\
+ \left[(t-\tau)\left(\frac{L^2}{4}+L+2\right)+\frac{1}{4}\right]  \itt |z_t|^2 \,dx\,ds.
\end{multline}
Let us introduce
\[
M(t) := \iO \sup_{\sigma\le t} |z (x,\sigma)|^2 \, dx
+  \iO \frac{|\nabla z(t)|^2}{2} \,dx .
\]
Due to (\ref{iterestimate}), whenever $(t-\tau)\bigl(\frac{L^2}{4}+L+2\bigr)+\frac{1}{4}\le 1$, i.e.
\begin{equation}\label{timestep}
\Delta t := t-\tau \le \frac{3}{4}\left(\frac{L^2}{4}+L+2\right)^{-1} ,
\end{equation}
we obtain
\begin{equation}\label{L2it}
M(t)  \le 3 M(\tau) + \itt f^2 \,dxds
\end{equation}

Next, we discretise the time interval $[0,T]$ by setting $t_k = k \Delta t$ for $0\le k\le K$,
where $\Delta t = T/K$ satisfies (\ref{timestep}).
Then, iteration of the estimate \eqref{L2it} yields
\begin{align}
M(T)  &\le 3 M(t_{K-1}) + \int_{t_{K-1}}^{T}\!\iO f^2 \,dxds
\le 3^2 M(t_{K-2}) + 3^1\int_{t_{K-2}}^{T}\!\iO f^2 \,dxds \nonumber\\
&\le 3^K M(0) + 3^{K-1}\int_{{0}}^{T}\!\!\iO f^2 \,dxds ,
\label{L2it2}
\end{align}
with
\[
 M(0) = \iO  |z_0 (x)|^2 dx
+  \iO \frac{|\nabla z_0|^2}{2} \,dx .
\]
This concludes the proof of \eqref{L2five}.
\end{proof}

\begin{proof}[Proof of Lemma \ref{pararegb}]

We shall now prove \eqref{Linfty}.
Recalling the parabolic remainder problem \eqref{remainderProblem},
we write the solutions in terms of the semi-group $e^{At}$
of the Laplace-operator $-\Delta$ subject
to the boundary conditions \eqref{remainBoun} and initial data $z_0\in L^{\infty}(\Omega)\cap H^1(\Omega)$, i.e.
\begin{align*}
z(x,t) &= e^{A t} z_0(x) + \int_0^{t} e^{A(t-s)} g(x,s)\,ds.
\end{align*}
By taking the supremum in space, we continue to estimate
for all $0\le t \le T$
\begin{multline*}
\|z(\cdot,t)\|_{L^{\infty}_x} \le \bigl\|e^{A t} z_0\bigr\|_{L^{\infty}_x} + \int_0^{t} \bigl\|e^{A(t-s)} g(\cdot,s)\bigr\|_{L^{\infty}_x}\,ds\nonumber\\
\le \|e^{A t}z_0\|_{L^{\infty}_x} + \int_0^t \bigl\|e^{A(t-s)}\bigr\|_{L^{\infty}_x\to L^{\infty}_x} \Bigl\|L\,
\sup_{\sigma\le s} | z(\cdot,\sigma) | + f(\cdot,s)\Bigr\|_{L^{\infty}_x} ds.
\end{multline*}
Next, we use Lemma \ref{Semi} that the operator norm $\|e^{A(t-s)}\|_{L^{\infty}_x\to L^{\infty}_x}\le 1$ for all $0\le s \le t \le T$ due to
the weak maximum principle for the heat equation subject to
the boundary condition \eqref{remainBoun}.
Thus,
\begin{align*}
\|z(\cdot,t)\|_{L^{\infty}_x} &\le \bigl\|z_0\bigr\|_{L^{\infty}_x}  +
L \int_0^t \bigl\|
\sup_{\sigma\le s} | z(\cdot,\sigma)| \bigr\|_{L^{\infty}_x} ds
 + \int_0^t \bigl\|f(\cdot,s)\bigr\|_{L^{\infty}_x} ds.
\end{align*}
Next, by taking the supremum in time for $t\le T$, we continue to estimate
\begin{align*}
\sup_{t\le T} \|z(\cdot,t)\|_{L^{\infty}_x} &\le \bigl\|z_0\bigr\|_{L^{\infty}_x}  + \int_0^T \bigl\|f(\cdot,s)\bigr\|_{L^{\infty}_x} ds
+ L \int_0^T
\sup_{\sigma\le s} \bigl\| z(\cdot,\sigma) \bigr\|_{L^{\infty}_x} ds.
\end{align*}
Therefore, a Gronwall Lemma for $\sup_{t\le T} \|z(\cdot,t)\|_{L^{\infty}_x}$ yields
\begin{align*}
\|z\|_{L^{\infty}_x(\Omega_{T})} =
\sup_{t\le T} \|z(\cdot,t)\|_{L^{\infty}_x} &\le
\left(\bigl\|z_0\bigr\|_{L^{\infty}_x}  + \int_0^T \bigl\|f(\cdot,s)\bigr\|_{L^{\infty}_x} ds\right) e^{LT},
\end{align*}
which proves \eqref{Linfty}.
\end{proof}

\begin{lem}\label{Semi}
Consider the heat equation
\begin{equation}\label{heat}
\begin{cases}
z_t - \Delta z =0,&\quad \text{on} \quad \Omega_T\\
\mathcal{B}[z]=0,&\quad \text{on} \quad \Gamma_T,\\
z(\cdot,0)=z_0\in L^{\infty}(\Omega),&\quad \text{on}\quad \Omega,
\end{cases}
\end{equation}
Then, the unique weak solution to \eqref{heat} propagates the $L^{\infty}$-norm
(as well as the non-negativity) of the initial data and the
associated semigroup satisfies
$
\|e^{At}\|_{L^{\infty}_x\to L^{\infty}_x}\le 1
$ for all $0\le  t $.
\end{lem}
\begin{proof}
The existence of a unique weak $H^1$-solution is well known, see e.g. \cite{Chi}. Note that general parabolic regularity for mixed boundary conditions $\mathcal{B}[z]=0$ only implies
$H^{3/2}_x$-smoothness, which is insufficient to yield $L^{\infty}_x$
bounds in space dimension $n\ge 3$.
The claims of the Lemma, however, are consequences of the
same arguments, which are used to prove the weak maximum principle, see e.g. \cite{Chi}.  For the sake of the reader we provide the details in the following.

First, we show the propagation of non-negativity of solutions subject to non-negative initial data $z_0\ge0$
by testing
$z_t = \Delta z$ with minus the negative part $-z^{-}=\min\{0,z\}$, which yields
with $\Gamma = \partial \Omega$ and $\nu$ being
the outer unit normal on $\Gamma$
$$
\iO z_t (-z^{-}) \,dx = \int_{\Gamma}
\nu\cdot\nabla z (-z^{-}) \,dA
- \iO \nabla z \cdot \nabla (-z^{-})\,dx
$$
and therefore, by using classical chain-rules arguments
for the negative part function (see e.g. \cite{Chi}) and
$\Gamma =  \Gamma_D \cup {\Gamma_N}$, $|{\Gamma_D}\cap{\Gamma_N}|=0$
\begin{equation*}
\frac{d}{dt} \iO \frac{[z^{-}]^2}{2} \,dx =
 -\int_{\Gamma_D} \nu\cdot\nabla z z^{-}\,dA
- \int_{\Gamma_N} \nu\cdot\nabla z z^{-}\,dA- \iO |\nabla z|^2 \mathbb{1}_{z\le 0}\,dx
\le 0,
\end{equation*}
where both boundary integrals vanish due to the boundary conditions.
Since $z_0^{-}=0$ a.e., this yields for all $t>0$ that
$z(x,t)\ge 0$ for a.a. $x\in\Omega$.

Next, we consider again non-negative initial data $z_0\ge0$.
Denoting $l=\|z_0\|_{L^{\infty}_x}$, we test
$(z-l)_t = \Delta (z-l)$ with the positive part $(z-l)^{+}$, which yields
$$
\iO (z-l)_t(z-l)^{+} \,dx = \int_{\Gamma}
\nu\cdot\nabla(z-l) (z-l)^{+} \,dA
- \iO \nabla (z-l) \cdot \nabla [(z-l)^{+}]\,dx
$$
and therefore
\begin{align*}
\frac{d}{dt} \iO \frac{[(z-l)^{+}]^2}{2} \,dx &=
 \int_{\Gamma_D} \nu\cdot\nabla(z-l) (z-l)^{+}\,dA
+ \int_{\Gamma_N} \nu\cdot\nabla z (z-l)^{+}\,dA\\
&\quad- \iO |\nabla (z-l)|^2 \mathbb{1}_{z\ge l}\,dx
\le 0,
\end{align*}
since both boundary integrals vanish. Together with $(z(0)-l)^{+}=0$ a.e., this yields due to the non-negativity of the solutions that for all $t>0$
\begin{equation*}
\|z(\cdot,t)\|_{L^{\infty}_x}\le l = \|z_0\|_{L^{\infty}_x}.
\end{equation*}

Finally, the statement of the Lemma for general initial data $z_0\in L^{\infty}(\Omega)$ follows
from superposing $z_0 = z_0^+ - z_0^-$ and applying the previous two steps to $z_0^+$ and $z_0^-$,
which implies altogether that
$\|e^{At}\|_{L^{\infty}_x\to L^{\infty}_x}\le 1$ for all $0\le  t $.
\end{proof}

\begin{proof}[Proof of Lemma \ref{higher}]
First, we recall that due to the embedding \eqref{yreg}, we have
$z\in L^2(\Omega;H^1(0,T))\cap L^{\infty}(0,T;V)\subset L^{q_0}(\Omega;C[0,T])$ for all $\theta_0\in(0,1/2)$ and $2<q_0 < \frac{2n}{n-2\theta_0}<\frac{2n}{n-1}$.

Next, we apply standard parabolic regularity estimates (see e.g. \cite[Theorem 7.20]{Lie})
that solutions to \eqref{remainderProblem} subject to a given
right-hand-side $g(z)\in L^{q_0}(\Omega_T)$ with $q_0>2$ and the mixed
homogeneous boundary data $\mathcal{B}[z]=0$
satisfy
$$
\|d_t\|_{L^{q_0}(\Omega_T)}, \|\Delta d\|_{L^{q_0}(\Omega_T)}\le \|g(z)\|_{L^{q_0}(\Omega_T)},
$$
which implies $d\in W^{1,q_0}(\Omega_T)$
for all $q_0 < \frac{2n}{n-1}$. Note that the exponent $\frac{2n}{n-1}$
corresponds to the limiting regularity in the case of mixed Dirichlet-Neumann
boundary conditions, i.e. $d\not \in H^{3/2}$ for $g\in L^2$,
but $d\in H^{3/2-\epsilon}$ for all $\epsilon>0$, see e.g. \cite{Sav}.

However, if $|\Gamma_N|=0$, we can bootstrap the above argument by using
$$
W^{1,q_0}(\Omega_T) \subset W^{\theta_1,q_0}(\Omega;W^{1-\theta_1,q_0}(0,T)).
$$
We aim to determine a $q_1>q_0$ such that
similar to \cite[page 266]{Vis}, we have
$$
W^{\theta_1,q_0}(\Omega;W^{1-\theta_1,q_0}(0,T)) \subset L^{q_1}(\Omega;C[0,T]),
\qquad 2<q_0<q_1.
$$
Hence, we chose $\theta_1\in(0,1)$ to satisfy $1-\theta_1-\frac{1}{q_0}>0$, i.e. $\theta_1 < \frac{q_0-1}{q_0}$ for $q_0>2$ and thus consider
$\theta_1\in(0,\frac{q_0-1}{q_0})$. Moreover, we set $\theta_1-\frac{n}{q_0}>-\frac{n}{q_1}$, i.e.
$$
q_1< \frac{q_0 n}{n - q_0\theta_1}, \qquad \text{provided that}\quad q_0 < \frac{n}{\theta_1},
$$
which is satisfied for $\theta_1\in(0,\frac{q_0-1}{q_0})$ chosen sufficiently small.
Then, we bootstrap this regularity argument, that is, we want to choose $q_{k+1}>q_{k}$ such that
$q_{k+1}< \frac{q_kn}{n-q_k\theta_{k+1}}$ provided that $n-q_k\theta_{k+1}>0$. In fact, the last condition is satisfied
by setting, for instance, $\theta_{k+1}:=\frac{n}{2q_k}>0$, which yields $q_{k+1}< 2q_k$ and we can choose
$q_{k+1}= \frac{3}{2} q_k$. Thus,
we obtain a sequence $q_k\nearrow+\infty$ (with $\theta_{k}\searrow 0$) as
$k\to\infty$. This finishes the proof of Lemma \ref{higher}.
\end{proof}
\medskip

\noindent{\bf Acknowledgements.} 	
The first author thanks Pavel Gurevich for some helpful suggestions
and the university of Graz for several opportunities to visit.
The second author acknowledges helpful discussions with Joachim Rehberg and
the kind hospitality of the Technical University of Munich.
The third author has been supported by the International Research Training
Group IGDK 1754 ``Optimization and Numerical Analysis for Partial Differential
Equations with Nonsmooth Structures'', funded by the German Research Council
(DFG) and the Austrian Science Fund (FWF).

\end{document}